\documentclass[10pt]{my_siamltex}
\usepackage{amsmath,amssymb}
\usepackage[mathcal]{euscript}

\usepackage{tikz}
\usetikzlibrary{patterns}
\definecolor{bfonce}{rgb}{0.,0.,0.8}	
\definecolor{bclair}{rgb}{0.87,0.92,1.}
\definecolor{orangec}{rgb}{1.,0.6,0.}
\usepackage{ifthen}
\newboolean{@long}\setboolean{@long}{true}

\newcommand{\dive}{{\mathrm{div}}}
\newcommand{\gradi}{{\boldsymbol \nabla}}

\newcommand{\bfu}{{\boldsymbol u}}

\newcommand{\bfr}{\boldsymbol r}
\newcommand{\bfx}{\boldsymbol x}

\newcommand{\xN}{\mathbb{N}}
\newcommand{\xR}{\mathbb{R}}
\newcommand{\ie}{{\em i.e.}}
\newcommand{\eg}{{\em e.g.}}
\newcommand{\bli}{\begin{list}{-}{\itemsep=1ex \topsep=1ex \leftmargin=0.5cm \labelwidth=0.3cm \labelsep=0.2cm \itemindent=0.cm}}

\title{Modelling of a spherical deflagration at constant speed}
\author{
D. Grapsas \thanks{Universit\'e d'Aix-Marseille ({\tt dionysis.grapsas@gmail.com})} \and
R. Herbin \thanks{Universit\'e d'Aix-Marseille ({\tt raphaele.herbin@univ-amu.fr})} \and
J.-C. Latch\'e \thanks{Institut de Radioprotection et de S\^{u}ret\'{e} Nucl\'{e}aire (IRSN) ({\tt jean-claude.latche@irsn.fr})} \and
Y. Nasseri \thanks{Universit\'e d'Aix-Marseille ({\tt youssouf.nasseri@univ-amu.fr})}
}
\begin{document}
\maketitle
\begin{abstract}
We build in this paper a numerical solution procedure to compute the flow induced by a spherical flame expanding from a point source at a constant expansion velocity, with an instantaneous chemical reaction.
The solution is supposed to be self-similar and the flow is split in three zones: an inner zone composed of burnt gases at rest, an intermediate zone where the solution is regular and the initial atmosphere composed of fresh gases at rest.
The intermediate zone is bounded by the reactive shock (inner side) and the so-called precursor shock (outer side), for which Rankine-Hugoniot conditions are written; the solution in this zone is governed by two ordinary differential equations which are solved numerically.
We show that, for any admissible precursor shock speed, the construction combining this numerical resolution with the exploitation of jump conditions is unique, and yields decreasing pressure, density and velocity profiles in the intermediate zone.
In addition, the reactive shock speed is larger than the velocity on the outer side of the shock, which is consistent with the fact that the difference of these two quantities is the so-called flame velocity, \ie\ the (relative) velocity at which the chemical reaction progresses in the fresh gases.
Finally, we also observe numerically that the function giving the flame velocity as a function of the precursor shock speed is increasing; this allows to embed the resolution in a Newton-like procedure to compute the flow for a given flame speed (instead of for a given precursor shock speed).
The resulting numerical algorithm is applied to stoichiometric hydrogen-air mixtures.
\end{abstract}
%
% --------------------------------------------------------------------
%
\section{Problem position}

\begin{figure}[h!]
\begin{center}
\scalebox{0.9}{
\begin{tikzpicture} 
\fill[color=orangec!80!white,opacity=0.3] (1.7,0) arc (0:180:1.7);
\fill[color=bclair,opacity=0.3] (5.5,0) -- (3.5,0) arc (0:180:3.5) -- (-5.5,0) arc (180:0:5.5);
\draw[very thick, color=red!80!black] (1.7,0) arc (0:180:1.7);
\draw[very thick, color=bfonce] (3.5,0) arc (0:180:3.5);
\node at (0, 0.6){burnt zone}; \node at (0, 0.25){(constant state)};
\node at (0, 2.55){intermediate zone}; \node at (0, 2.2){(regular solution)};
\node at (0, 4.5){unburnt zone}; \node at (0, 4.15){(constant initial state)};
\node at (1, 1){$W_b$};
\node at (1.37, 1.37){$W_2$};
\node at (2.3, 2.3){$W_1$};
\node at (2.67, 2.67){$W_0$};
\draw[very thick, color=red!80!black] (-5.5,-1) -- (-4.3,-1); \node[anchor=west] at (-4.1, -1){reactive shock, $r=\sigma_r\,t$.};
\draw[very thick, color=bfonce] (-5.5,-1.4) -- (-4.3,-1.4); \node[anchor=west] at (-4.1, -1.4){precursor shock, $r=\sigma_p\,t$.};
\node[anchor=west] at (1., -1){$W=(\rho,u,p)$: local fluid state.};
\end{tikzpicture}
}
\caption{Structure of the solution.}
\label{Fig.1}
\end{center}
\end{figure}
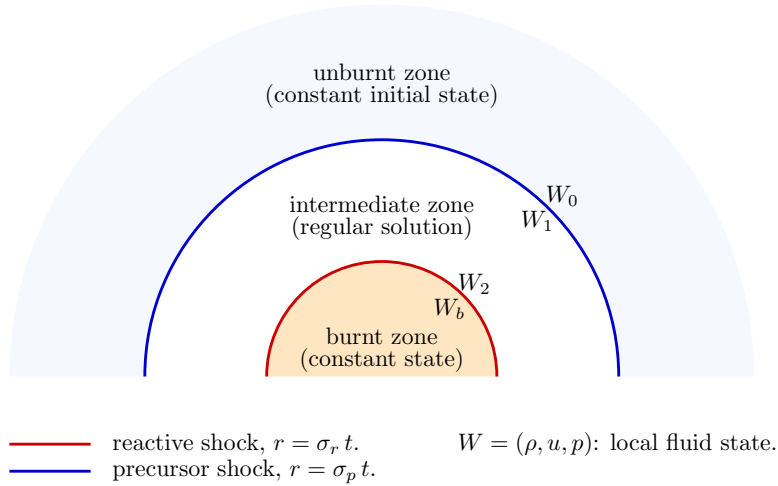

We address the flame propagation in a reactive infinite atmosphere of initial constant composition.
The ignition is supposed to occur at a single point (chosen to be the origin of $\xR^3$) and the flow is supposed to satisfy a spherical symmetry property: the density $\rho$, the pressure $p$, the internal energy $e$ and the entropy $s$ only depend on the distance $r$ to the origin and the velocity reads $\bfu=u \bfr/r$, where $\bfr$ stands for the position vector.
The flame is supposed to be infinitely thin and to move at a constant speed.
The flow is governed by the Euler equations, and we seek a solution with the following structure:
\bli
\item the solution is self-similar, \ie\ the quantities $\rho$, $p$, $e$, $s$ and $u$ are functions of the variable $x=r/t$ only.

\item the flow is split in three zones, referred to as the inner, intermediate and outer zones.
The inner zone stands for the burnt zone while, in the other two zones, the gas is supposed to be in its initial (referred to as fresh or unburnt) composition.
Burnt and fresh gases differ by the expression of the total energy:
\begin{equation}\label{eq:energies}
E = \frac 1 2  u^2 + e - \zeta_b\, Q,\quad \zeta_b=1 \mbox{ in the burnt zone, } \zeta_b=0 \mbox{ in the fresh zone,}
\end{equation}
with $Q >0$ the chemical heat reaction.
Both burnt and unburnt gases are considered as ideal gases, possibly with different heat capacity ratios: 
\[
p=(\gamma - 1) \rho e,\quad \gamma=\gamma_b \text{ for burnt gases},\ \gamma=\gamma_u \text{ for unburnt gases}.
\]

\item In the burnt zone, the solution is supposed to be constant; this constant state is denoted by $W_b=(\rho_b,u_b,p_b)$.
For symmetry reasons, the fluid is at rest in this zone, \ie\ $u_b=0$. 

\item The burnt and intermediate zones are separated by a shock, which coincides with the flame front.
This shock is called the reactive shock, and travels at a constant speed $\sigma_r$.
The outer state of the shock is denoted by $W_2=(\rho_2,u_2,p_2)$.
Note that the usual Rankine-Hugoniot jump conditions apply at the reactive shock, up to the fact that the expression of the total energy in the inner and outer states differ (see Equation \eqref{eq:energies}).

\item The intermediate and the outer zone are separated by a 3-shock, referred to as the precursor shock, and travelling at a velocity denoted by $\sigma_p$.
We denote by $W_1=(\rho_1,u_1,p_1)$ the inner state of the precursor shock, and, since the usual jump conditions for the Euler equations apply, we have $\sigma_p \geq u_1$.
In the outer zone, conditions are constant and equal to the initial condition $W_0=(\rho_0,u_0,p_0)$; the fluid is at rest, \ie\ $u_0=0$.

\item In the intermediate zones, the states $W_2$ and $W_1$ are supposed to be linked by a regular solution.
\end{list}
In addition, for physical reasons, we expect that
\begin{equation}\label{eq:uf}
u_2 >0 \mbox{ and } \sigma_r = u_2 + u_f \mbox{ with  }u_f >0.
\end{equation}
Indeed, the velocity $u_f$ is the velocity at which the chemical reaction progresses in the fresh gases; these are pushed away from the origin by the expansion of the burnt gases, and therefore $u_2 > 0$.

\medskip
The aim of this paper is to build a numerical procedure to compute a solution with the above described structure.
More precisely speaking, we present the two following developments:
\bli
\item First, for a given precursor shock speed $\sigma_p$, we derive a solution with the desired structure in a constructive way (and this construction yields a unique solution), and propose a simple numerical scheme to compute it.
Moreover, the constructed solution is such that the inequalities \eqref{eq:uf} are satisfied (in fact, we obtain that $\sigma_r - u_2 >0$ since we seek and find a solution such that $x-u(x) >0$ in the whole intermediate zone), and thus yields a physically meaningful flame velocity $u_f$.
\item As a by-product, we numerically obtain the velocity $u_f$ as a function of $\sigma_p$, \ie\ we construct a function $\widetilde {\mathcal G}$ such that $u_f=\widetilde {\mathcal G}(\sigma_p)$, and observe that this function $\widetilde {\mathcal G}$ is strictly increasing, which was expected from physical reasons (the faster the combustion, the stronger the generated shock-wave).
It is thus easy to build an iteration to compute the flow associated to a given $u_f$, which is generally the problem of physical interest.
\end{list}
Finally, this process is implemented in the free CALIF$^3$S software \cite{califs} developped at the french Institut de Radioprotection et de S\^uret\'e Nucl\'eaire (IRSN) and applied to obtain solutions as a function of $u_f$ for a stoichiometric mixture of hydrogen and air.

\medskip
The derivation of a solution for the same problem may be found in \cite{kuh-73-pre}; however, the techniques used in this latter paper are different (the solution is performed in the phase space), and the uniqueness of the construction together with the proof of the decreasing properties of the solution are not explicit.
The developments in \cite{kuh-73-pre} are built upon techniques developed for non-reactive problems in \cite{sed-45-onc,tay-46-air}.
Approximate solutions in closed form are given in \cite{gui-76-pre,cam-77-eco}, and extensions to accelerating flames may be found in \cite{str-79-bla,des-81-flo}.
Finally, the complete solution of the plane case (\ie\ the one-dimensional case in cartesian coordinates), in closed form, is given in \cite{bec-10-rea}.
%
% -----------------------------------------------------------------------------------------------------------------------------------------------------------
%
\section{Solution for a given precursor shock speed}
We first propose a constructive derivation of the solution for a given precursor shock speed, and then state the numerical scheme to compute it.
%
% --------------------------------------------------------
%
\subsection{Derivation of the solution}
Since the fluid state in the outer zone $W_0$ is given, we may equivalently use hereafter either $\sigma_p$ or the precursor shock Mach number defined by $M_p=\sigma_p/c_0$, with $c_0=(\gamma_u\,p_0/\rho_0)^{1/2}$ the speed of sound in the outer zone.
We recall that, for entropy condition reasons, $M_p > 1$.
By a standard computation, the Rankine-Hugoniot conditions at the precursor shock yields the state $W_1$:
\begin{equation}\label{eq:state_1}
\rho_1= \frac{\gamma_u +1}{\gamma_u -1+\dfrac 2 {M_p^2}}\ \rho_0,
\quad u_1=(1-\frac{\rho_0}{\rho_1})\, \sigma_p,
\quad p_1=p_0+(1-\frac{\rho_0}{\rho_1})\,\rho_0\,\sigma_p^2.
\end{equation}
In addition, using the fact that the fluid is at rest in the burnt zone, we show in Appendix \ref{sec:RS} that State 2 must satisfy the following relation:
\begin{equation}\label{eq:rel_state_2}
\mathcal F_r(\sigma_r):=
\frac12 u_2^2 + \frac 1 {\gamma_b -1} u_2\, \sigma_r + \bigl( \frac{\gamma_u}{\gamma_u -1} - \frac{\gamma_b}{\gamma_b -1} \dfrac{\sigma_r}{\sigma_r - u_2} \bigr) \dfrac{p_2}{\rho_2} + Q=0,
\end{equation}
where $\sigma_r$ stands for the (unknown) reactive shock speed.
In addition, the same jump conditions give the burnt state $W_b$ as a function of $W_2$:
\begin{equation}\label{eq:state_b}
\rho_b=\rho_2 (\dfrac{\sigma_r - u_2}{\sigma_r}),\quad p_b = p_2 - \rho_2 u_2 (\sigma_r - u_2).
\end{equation}
To complete the derivation of the solution, we must now show that the following program makes sense: starting from $x=\sigma_p$, solve the Euler equations for $x \leq \sigma_p$ until the point $x=\sigma_r$ where Equation \eqref{eq:rel_state_2} is verified.
The solution at this point is equal to $W_2$ and Equations \eqref{eq:state_b} yield the burnt state $W_b$.
Let us now embark on this development.
%
% ------------------------------------------
%

\medskip
\paragraph{Governing equations in the intermediate zone}
Since we suppose that the solution is regular in this zone, we may replace the total energy balance in the Euler equations by the entropy equation, which, under the spherical symmetry assumption, yields the following system:
\begin{subequations}
\begin{align} \label{eq:euler_mass} &
\partial_t (r^2 \rho ) + \partial_r (r^2 \rho u ) =0,
\\ \label{eq:euler_mom} &
\partial_t (r^2\rho u) +  \partial_r (r^2(\rho u^2+p)) = 2r p,
\\ \label{eq:euler_ent} &
\partial _t (r^2\rho s ) + \partial _r (r^2\,\rho s u ) = 0.
\end{align}\end{subequations}
The mass balance equation \eqref{eq:euler_mass} may be developed to obtain:
\begin{equation} \label{eq:mass_nc}
\partial_t \rho + u\, \partial_r \rho + \rho\, \partial_r u + \frac 2 r \rho u=0.
\end{equation}
In addition, thanks to the mass balance equation and for a regular function $f=f(t,r)$, we have:
\[
\partial_t (r^2 \rho f) + \partial_r (r^2 \rho f u) = \rho\, (\partial_t f + \rho u\, \partial_r f).
\]
Using this identity in the momentum and entropy balances, \ie\ Equation \eqref{eq:euler_mom} and \eqref{eq:euler_ent} respectively, we get:
\begin{equation} \label{eq:mom_s_nc}
\begin{array}{l}\displaystyle
\partial_t u +  u\, \partial_r u + \frac 1 \rho\, \partial_r p = 0,
\\[1.5ex]
\partial_t s + u\, \partial_r s =  0.
\end{array}
\end{equation}
We now use the fact that, if a regular function $\varphi(t,r)$ only depends on $x=r/t$, which means that there exists $\tilde\varphi:\ \xR \rightarrow \xR$ such that $\tilde\varphi(x)=\varphi(t,r)$, we have
\[
\partial_t \varphi(t,r) = -\frac r {t^2}\, \tilde\varphi'(x) \mbox{ and } \partial_r \varphi(t,r) = \frac 1 t\, \tilde\varphi'(x).
\]
Since we look for a self-similar solution, we may apply this identity to \eqref{eq:mass_nc} and \eqref{eq:mom_s_nc}.
Keeping the same notation for functions of the pair $(t,r)$ and $x$ for short, we obtain the following system:
\begin{subequations}\label{eq:euler_ss}
\begin{align}\label{eq:mass_ss} &
\dfrac{-x + u}{\rho}\, \rho'(x) + u'(x) + \dfrac{2u(x)}{x} = 0,
\\ \label{eq:mom_ss} &
(-x +u(x))\, u'(x) + \dfrac 1 \rho(x)\, p'(x) = 0,
\\ \label{eq:ent_ss} &
(u(x)-x)\, s'(x)= 0.
\end{align}
\end{subequations}
Let us now suppose that $u <x$ in the intermediate zone.
Note that the fact that the precursor shock is a 3-shock implies that $u_1<\sigma_p$, so the assumed inequality is true in the outer boundary of the intermediate zone, and the assumption amounts to suppose that the intermediate zone ends (more precisely speaking, may be made to end in the construction of the solution) before $u=x$ occurs, which will be checked further.
The last relation thus implies that the entropy remains constant over the zone:
\begin{equation}
s=\frac p {\rho^{\gamma_u}} = s_1 = \frac {p_1} {\rho_1^{\gamma_u}},
\end{equation}
and this is a known value thanks to \eqref{eq:state_1}.
We thus have $p'=\gamma_u\, s_1\, \rho^{\gamma_u-1}\rho'$; using $c^2=\gamma_u\, p/\rho =\gamma_u\, s_1\, \rho^{\gamma_u-1}$, we thus get $p'=c^2\,\rho'$.
Substituting this expression in \eqref{eq:mass_ss}-\eqref{eq:mom_ss} and solving for $\rho'$ and $u'$, we get:
\begin{subequations}\label{eq:ode}
\begin{align}\label{eq:ode_rho} &
\rho'(x)=-\frac{2 u(x) (u-x)}{x\bigl((u(x)-x)^2-c(x)^2 \bigr)}\ \rho,
\\ \label{eq:ode_u} &
u'= \frac{2 c(x)^2}{x\bigl((u(x)-x)^2-c(x)^2 \bigr)}\ u.
\end{align}
\end{subequations}
This system of coupled ODEs is complemented by initial conditions, which consist in the data of the velocity and the density at the precursor shock, \ie\ at the outer boundary of the intermediate zone $x=\sigma_p$:
\begin{equation}
	\rho(\sigma_p)=\rho_1, u(\sigma_p)=u_1.
	\label{condini}	
\end{equation}

%
% ------------------------------------------
%

\bigskip
\paragraph{Existence, uniqueness and properties of the solution}
We begin by proving an {\em a priori} property of the solution, namely the fact that $\rho(x)$ and $u(x)$ are necessarily decreasing functions in the intermediate zone.
To this end, we will invoke the following easy lemma, which is a consequence of the mean value theorem.

\smallskip
\begin{lemma}\label{lmm:tl}
Let $h$ be a continuously differentiable real function, let us suppose that there exists $a > 0$ such that $h(a) > 0$, and that $h$ satisfies the property $h'(x) \leq 0$ if $h(x) > 0$.
Then $h(x) \geq h(a)$, for all $x \leq a$.
\end{lemma}

\medskip
From the expression \eqref{eq:state_1}, we know that $0 < u_1 < \sigma_p$ and $\rho_1 >0$.
Let us now introduce $\sigma_\ell$ as the largest real number in $[0,\sigma_p)$ such that, for $x \in [\sigma_\ell,\sigma_p]$, $0 \leq u \leq x$ and $\rho \geq 0$.
Note that such a closed interval exists by the continuity (assumed in this zone) of $\rho$ and $u$. 
We are now in position to state the following result.

\smallskip
\begin{lemma}[Variations of the solution] \label{lmm:var}
Let us suppose that the pair $(\rho, u)$ satisfies \eqref{eq:ode}.
Then $\rho$ and $u$ are two decreasing functions over $[\sigma_\ell,\sigma_p]$.
Consequently, $\rho \geq \rho_1 >0$ and $u \geq u_1 >0$ over $[\sigma_\ell,\sigma_p]$.
\end{lemma}
\begin{proof}
Let us consider the function $h: \xR_+ \in \xR$ defined by $h(x) = u(x) + c(x) - x$.
First, we remark that, by the Lax entropy condition, we have $h(\sigma_p)=u_1+c_1-\sigma_p > 0$ (and this property may be checked using the expressions \eqref{eq:state_1} of $W_1$).
Second, if $h(x)>0$, since by assumption  $u(x) \leq x$, $\rho \geq 0$ and $c \geq 0$, we have:
\[
(u(x)-x)^2-c(x)^2=(u(x)-x-c(x))\,(u(x)-x+c(x)) \leq 0.
\]
Equations \eqref{eq:ode_rho} and \eqref{eq:ode_u} thus readily imply that $\rho'\leq 0$ and $u' \leq 0$ since, still by assumption, $u \geq 0$ and $\rho \geq 0$.
The function $h$ is thus the sum of three non-increasing functions, and is hence non-increasing itself.
Lemma \ref{lmm:tl} applies, and yields $h(x) =u(x)-x-c(x) \geq h(\sigma_p) >0$ over the whole interval $[\sigma_\ell,\sigma_p]$, which in turn implies $\rho' \leq 0$ and $u' \leq 0$.
We thus have $\rho \geq \rho_1 >0$ and $u \geq u_1 >0$, which finally yields $\rho' < 0$ and $u' < 0$ over $[\sigma_\ell,\sigma_p]$.\\
\end{proof}

\medskip
Note that the inequality $h(x) \geq h(\sigma_p)> 0$ derived in this proof implies that the denominator in Equations \eqref{eq:ode_rho} and \eqref{eq:ode_u} does not vanish in the interval $[\sigma_\ell,\sigma_p]$.
The right-hand side of System \eqref{eq:ode} is thus a $C^\infty$ function of $\rho$, $u$ and $x$, and the existence and uniqueness of a solution follows by the Cauchy-Lipschitz theorem.
This result is stated in the following lemma.

\smallskip
\begin{lemma}[Existence and uniqueness of the solution]
There exists one and only one solution $(\rho, u)$ of System \eqref{eq:ode}-\eqref{condini} over the interval $[\sigma_\ell, \sigma_p]$. 
\end{lemma}

\medskip
In addition, Lemma \ref{lmm:var} allows to characterize $\sigma_\ell$.
Indeed, since $u \geq u_1 >0$ and $\rho \geq \rho_1 >0$ over $[\sigma_\ell,\sigma_p]$, by definition of $\sigma_\ell$, either $\sigma_\ell=0$ or $u(\sigma_\ell)=\sigma_\ell$.
Since $u \geq u_1 >0$ and $u(x)\geq x$ over $[\sigma_\ell,\sigma_p]$, the first option cannot hold, and we get
\begin{equation}\label{eq:sigmaL}
u(\sigma_\ell)=\sigma_\ell.
\end{equation}

\medskip
To complete the construction of a solution, it now remains to show the existence of a real number $\sigma_r$, \ie\ the fact that there exists $x \in (\sigma_\ell,\sigma_p)$ such that $W(x)=\bigl(\rho(x),u(x),p(x)\bigr)$ satisfies the condition $\mathcal F_r(x)=0$, where $\mathcal F_r$ is given in \eqref{eq:rel_state_2}:
\begin{equation}\label{eq:rel_state_2_x}
\mathcal F_r(x)=
\frac12 u(x)^2 + \frac 1 {\gamma_b -1} x\,u(x) + \bigl( \frac{\gamma_u}{\gamma_u -1} - \frac{\gamma_b}{\gamma_b -1} \dfrac{x}{x - u(x)} \bigr) \dfrac{p(x)}{\rho(x)} + Q=0.
\end{equation}
The existence of $\sigma_r$ is stated in the following lemma.

\smallskip
\begin{lemma}[Existence of $\sigma_r$]
The function $\mathcal F_r$ is defined and continuously differentiable over $(\sigma_\ell, \sigma_p]$,  $\mathcal F_r(\sigma_p)>0$ and $\lim_{x \rightarrow \sigma_\ell^+}\mathcal F_r(x)=-\infty$. 
Consequently, the set $\mathcal S_r=\{x \in (\sigma_\ell, \sigma_p) \mbox{ such that } \mathcal F_r(x)=0\}$ is a non-empty closed subset of $(\sigma_\ell, \sigma_p)$ which admits a maximal element $\sigma_r$. 
\end{lemma}
\begin{proof}
When $x$ tends to $\sigma_\ell$, we have seen that $u(\sigma_\ell)$ tends to $\sigma_\ell$ and thus $\mathcal F_r$ tends to $-\infty $.
When $x = \sigma_p$, thanks to the inequalities $u_1 >0$ and $\sigma_p - u_1 > 0$, we observe that $\mathcal F_r(\sigma_p)$ is an addition of positive terms:
\[
\mathcal F_r(\sigma_p)=\frac 1 2 u_1^2 + \frac{\sigma_p}{\gamma -1} u_1  + \frac{1}{\gamma -1} \big( 1 -  \dfrac{\sigma_p}{\sigma_p - u_1} \big) c_1^2 + Q > 0.
\]
\end{proof}

\medskip
In addition, when $\gamma_u=\gamma_b$, we are able to prove that $\mathcal F_r(\sigma_p)$ is an increasing function over $(\sigma_\ell, \sigma_p)$, and therefore the set $\mathcal S_r$ contains a single point; the proof of this result is given in Appendix \ref{sec:fr}.

\medskip
Finally, note that $\sigma_r > \sigma_\ell$; since $u$ is a decreasing function, this yields that $\sigma_r-u(\sigma_r)>0$.
As mentioned in \eqref{eq:uf}, this was expected, from a physical point of view, since this quantity is nothing else that the flame velocity $u_f$.
%
% --------------------------------------------------------
%
\subsection{Numerical approximation of the solution in the intermediate zone}
The problem tackled in this section is twofold: first, we need to solve numerically the system of ODEs \eqref{eq:ode}-\eqref{condini}, and second, to determine the speed of the reactive shock $\sigma_r$.
To this purpose, we solve \eqref{eq:ode}-\eqref{condini} by an explicit Euler scheme, starting at $N \in \xN$ and $x^N=\sigma_p$ and, for indices $n$ decreasing from $N$, performing steps of $-\delta x$, with $\delta x =\sigma_p/N$; at each new step $n$ associated to $x^n=n \delta x$, we obtain $W^n$ and we evaluate the function $\mathcal F_r$, until we obtain $\mathcal F_r(x^n) \leq 0$.
Here, the algorithm stops and we know that the computed approximation $\sigma^{\rm app}_r$ of $\sigma_r$ satisfies $x^n < \sigma^{\rm app}_r < x^{n+1}$; for $\delta x$ small enough, $x^{n+1}$ may thus be considered as a reasonable approximation of $\sigma_r$; this is indeed the way it is computed in the numerical experiments described below.

\medskip
The scheme thus reads:
\begin{equation}\label{eq:ode_scheme}
\begin{array}{l}
\mbox{for } n=N,\ u^N = u_1,\ \rho^N = \rho_1,
\\[3ex]
\mbox{for } n=N-1 \mbox{ to } 0  \mbox{ and while } \mathcal F_r(x^{n+1}) > 0,
\\[2ex] \displaystyle \hspace{5ex}
(c^{n+1})^2 = \gamma \, s_1 \, (\rho^{n+1})^{\gamma_u-1},
\\[2ex] \displaystyle \hspace{5ex}
\rho^n = \rho^{n+1 }+ \delta x\ \dfrac{2\,u^{n+1}\,(u^{n+1}-x^{n+1})}{x^{n+1}\,\bigl( (u^{n+1}-x^{n+1})^2 - (c^{n+1})^2 \bigr)}\ \rho^{n+1},
\\[3ex] \displaystyle \hspace{5ex}
u^n = u^{n+1} - \delta x\ \dfrac{2\,(c^{n+1})^2}{x^{n+1}\,\bigl( (u^{n+1}-x^{n+1})^2 - (c^{n+1})^2 \bigr)} u^{n+1}.
\end{array}
\end{equation}
Then, for any valid value of $n \leq N$, the pressure is given by
\[
p^n = s_1\,(\rho^n)^{\gamma_u}.
\]
Since the algorithm stops as soon as $\mathcal F_r(x^{n+1})$ becomes negative, from the expression of this latter function, we have $u^n < x^n$ in all the performed steps $n$.
We thus have $u^n > 0$, $\rho^n > 0$, $ u^n \geq u^{n+1}$ and $\rho^n \geq \rho^{n+1}$ at all steps.
%
% -----------------------------------------------------------------------------------------------------------------------------------------------------------
%
\section{Solution for a given flame speed}
The construction performed in the previous section shows that, to any precursor shock velocity $\sigma_p$ greater than the speed of sound $c_0$ in the outer zone of the fresh atmosphere, we are able to associate a positive flame velocity $u_f$ given by $u_f=\sigma_r-u_2$.
In addition, even if we have no proof, physical arguments suggest that $u_f$ is an increasing function of $\sigma_p$ (or  equivalently of the Mach number $M=\sigma_p/c_0$, considering a family of problems with the same initial atmosphere and thus $c_0$ as a fixed parameter); this behaviour is confirmed by numerical experiments (see Section \ref{sec:num}).
Computing the flow for a given $u_f$, which is in fact usually the engineering problem to be tackled, amounts to invert the function $u_f =\widetilde {\mathcal G}(M)$, and this equation for $M$ thus should have one and only one solution, at least for reasonable values of $u_f$.
To compute this solution, we define $\mathcal G$ by
\begin{equation}
	\label{nu}
	\mathcal G(M): =   \widetilde {\mathcal G} (M) - u_f
\end{equation}
and search for $M$ such that $\mathcal G(M)=0$ with the following iterative algorithm depending on the parameters $M_0$, $\delta$ and $\epsilon$:
\[
\begin{array}{ll}
\mbox{initialization:}
&
\mbox{let } M_0 \mbox{ be given, and compute } \mathcal G(M_0),
\\[1ex] &
\mbox{let } M_1=M_0 + \delta,\mbox{ and compute } \mathcal G(M_1)
\\[2ex]
\mbox{current iteration:}
&
\text{ For } k \geq 2, \mbox{ let }M_k = M_{k-1} - \dfrac{M_{k-1} - M_{k-2}}{\mathcal G(M_{k-1}) - \mathcal G(M_{k-2})} \mathcal G(M_{k-1}),
\\ &
\mbox{ and compute } \mathcal G(M_k).
\\[2ex]
\mbox{stopping criteria:}
&
\mbox{ stop when } \mathcal G(M_k) \leq \epsilon.
\end{array}
\]
This algorithm is used in the following section with $M_0=1.0001$, $\delta=0.001$ and $\epsilon=10^{-5}$.
Convergence is obtained for all cases, provided that the number of cells used in the numerical computation of the solution in the intermediate zone is large enough; otherwise, the error on $\sigma_r$ is too large and the prescribed tolerance threshold for the value of $\mathcal G$ cannot be reached.

%
% -----------------------------------------------------------------------------------------------------------------------------------------------------------
%

\begin{figure}[tb]
\scalebox{0.9}{
\begin{tikzpicture} 
\node[anchor=east] at (10., 10){\includegraphics[scale=0.58]{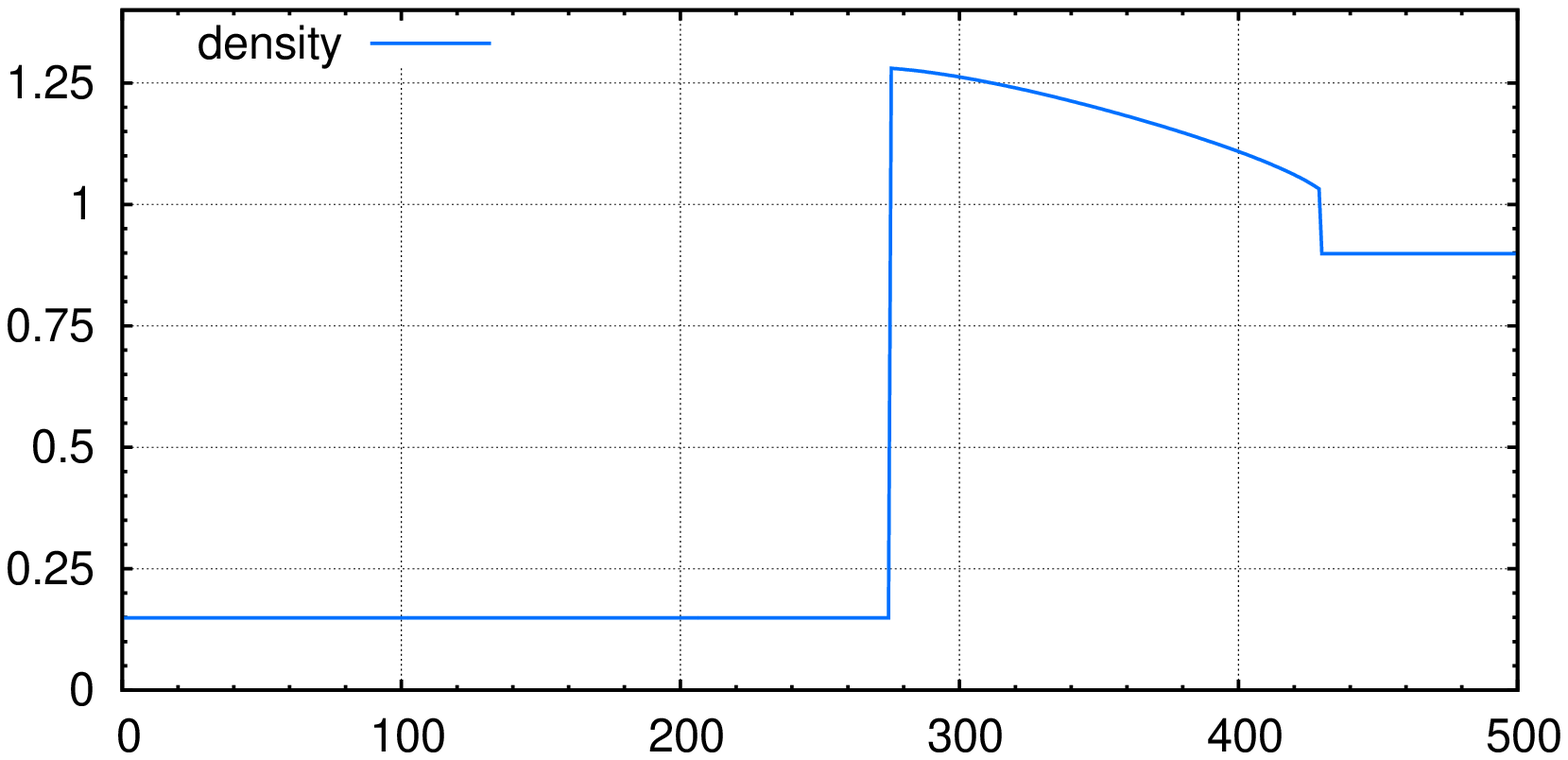}};
\node[anchor=east] at (10., 5){\includegraphics[scale=0.57]{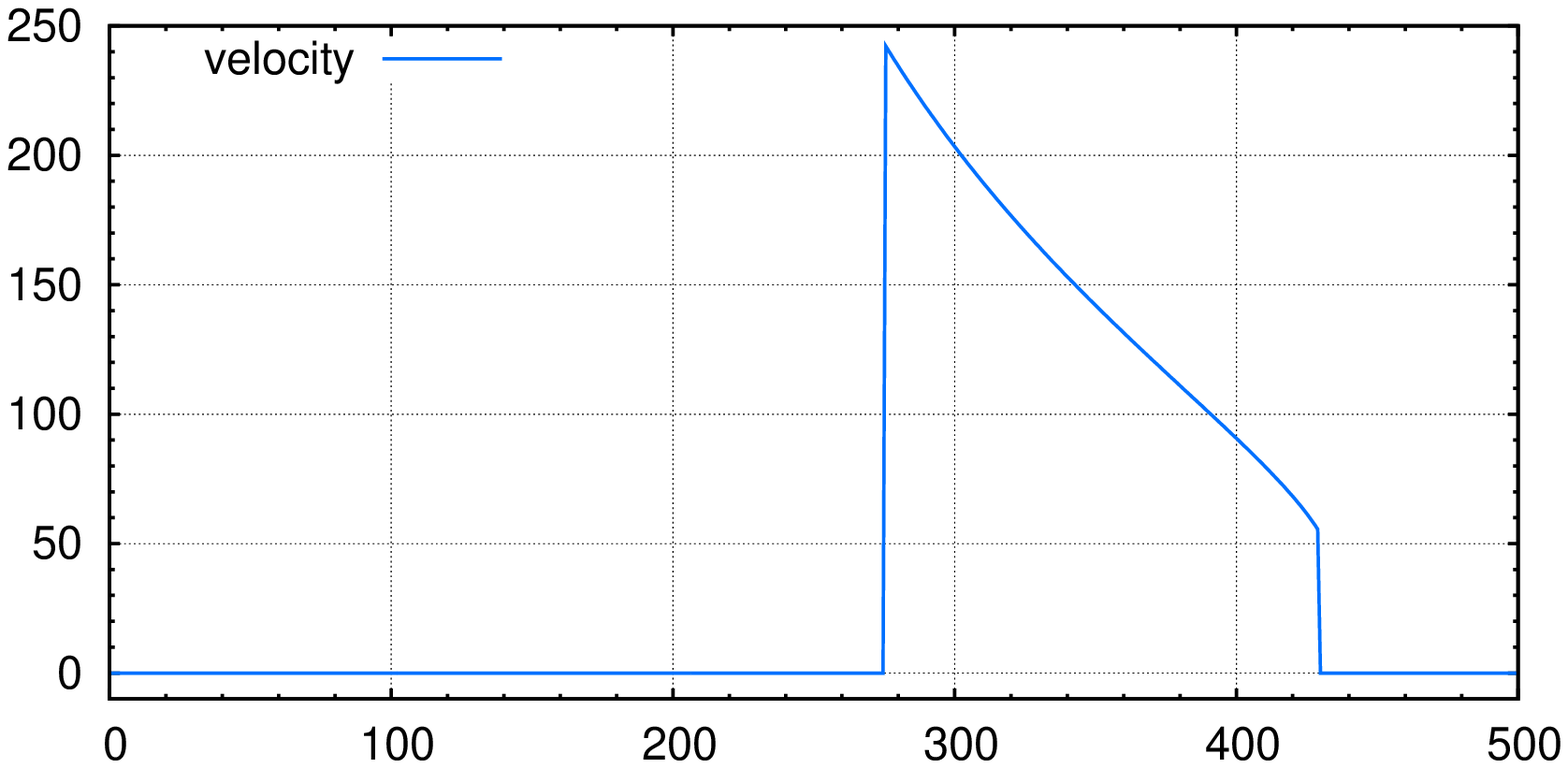}};
\node[anchor=east] at (10., 0){\includegraphics[scale=0.6]{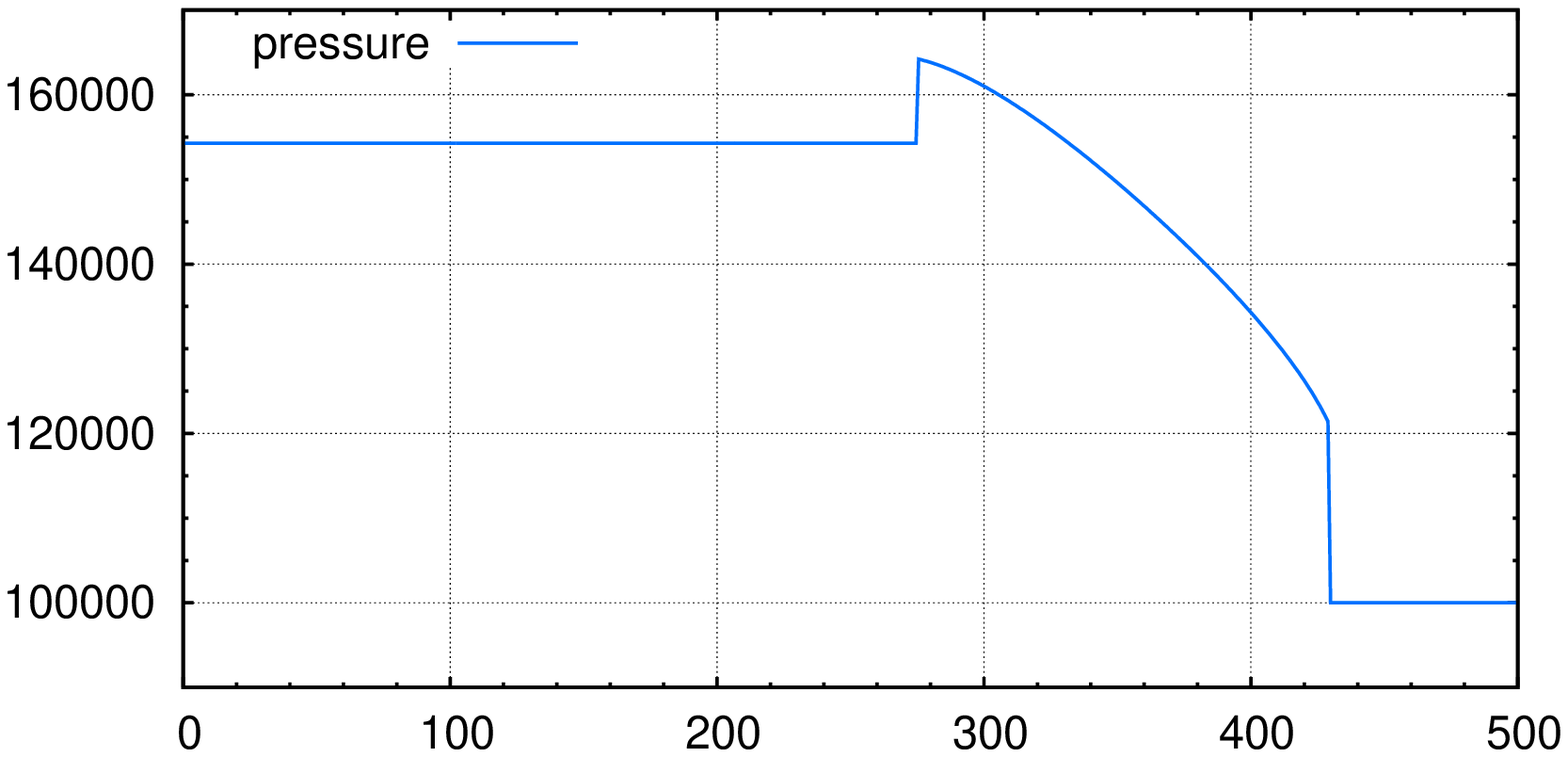}};
\end{tikzpicture}
}
\caption{Density ($\mathrm{kg}\,\mathrm{m}^{-3}$), velocity ($ \mathrm{m}\,\mathrm{s}^{-1}$) and pressure ($\mathrm{Pa}$) profiles obtained for a flame velocity of $u_f=32$\,m/s.}
\label{fig:uf=32} 
\end{figure}

\begin{figure}[tb]
\scalebox{0.9}{
\begin{tikzpicture} 
\node[anchor=east] at (10., 10){\includegraphics[scale=0.58]{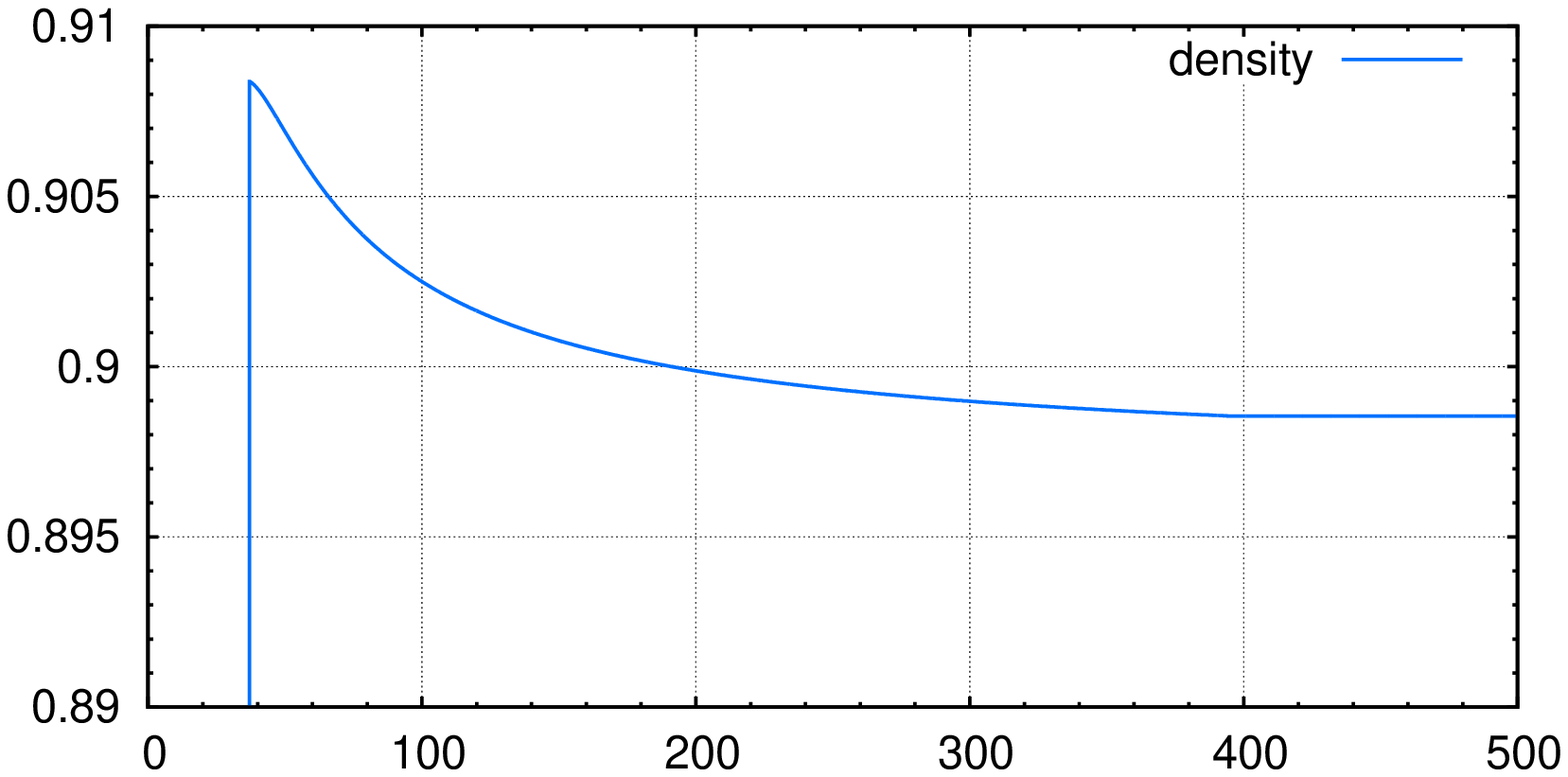}};
\node[anchor=east] at (10., 5){\includegraphics[scale=0.57]{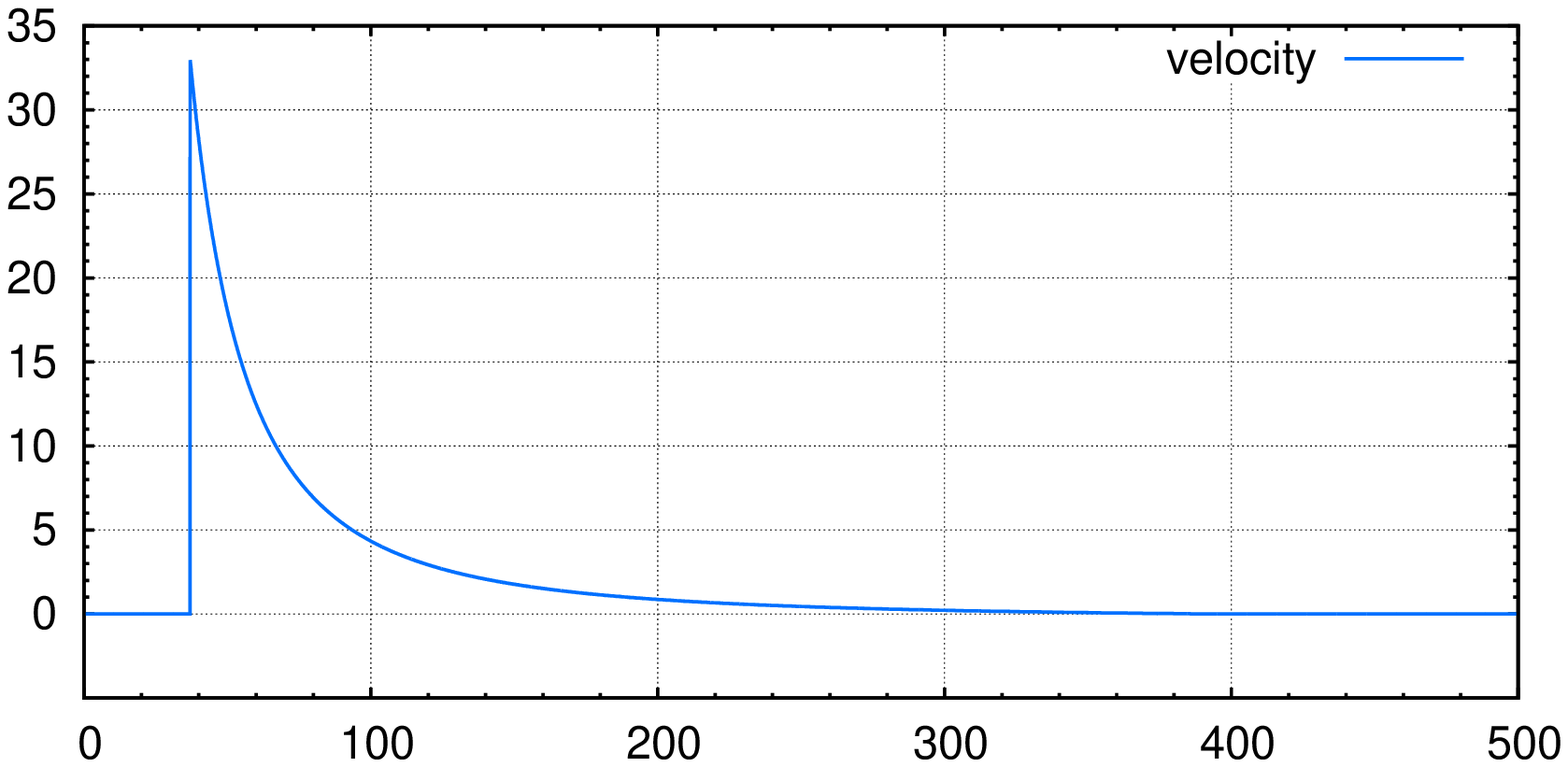}};
\node[anchor=east] at (10., 0){\includegraphics[scale=0.6]{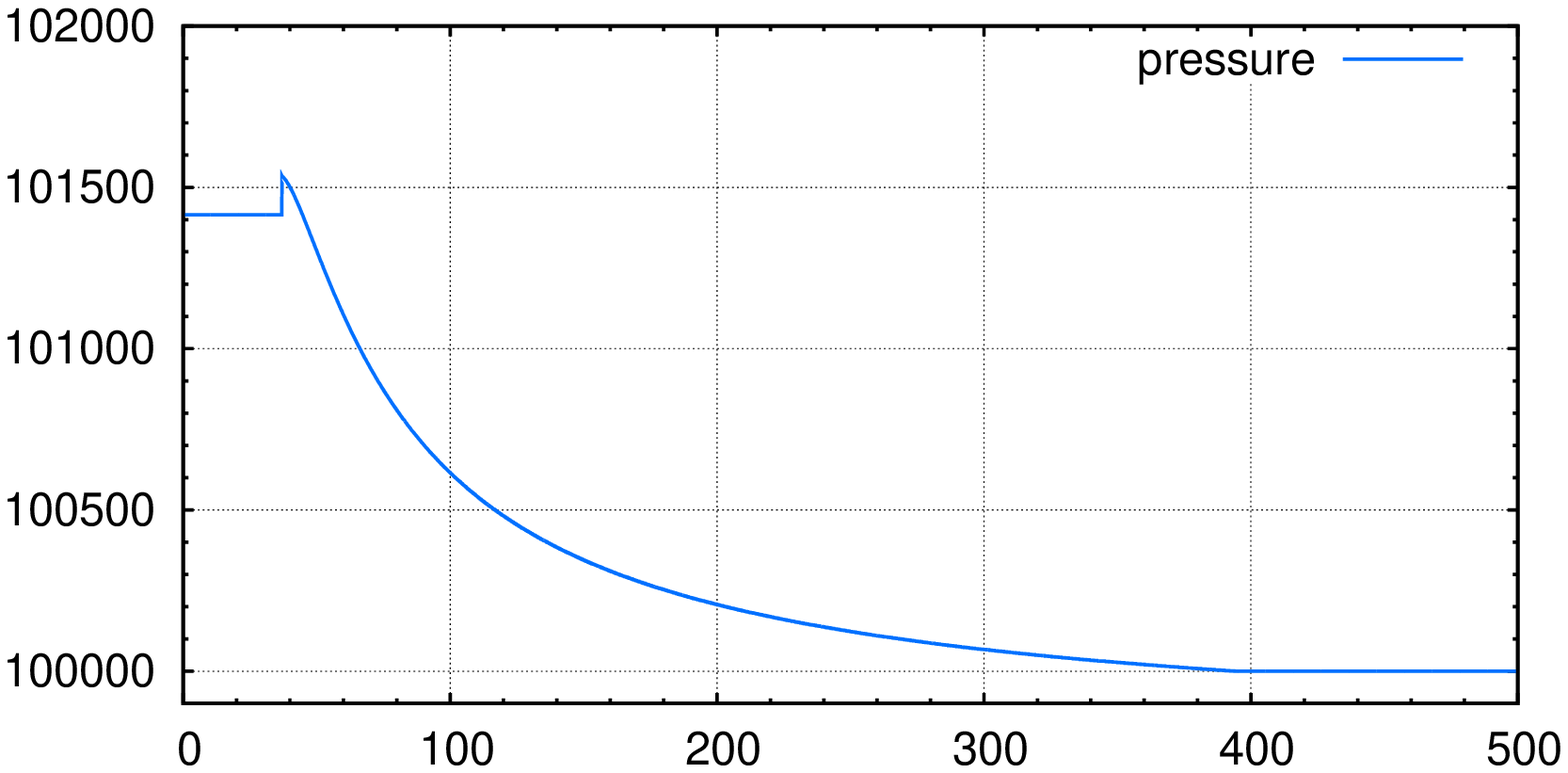}};
\end{tikzpicture}
}
\caption{Density ($\mathrm{kg}\,\mathrm{m}^{-3}$), velocity ($ \mathrm{m}\,\mathrm{s}^{-1}$) and pressure ($\mathrm{Pa}$) profiles obtained for a flame velocity of $u_f=4$\,m/s.}
\label{fig:uf=4} 
\end{figure}

\section{Application to hydrogen deflagrations}\label{sec:num}
We now apply the developed procedure for a flame propagating in a stoichiometric mixture of hydrogen and air.
We consider a unique total and irreversible chemical reaction, which reads:
\[
2\, \mathrm{H}_2 + \mathrm{O}_2 \longrightarrow 2\, \mathrm{H}_2\mathrm{O}
\]
Supposing that air is composed of $1/5$ of oxygen and $4/5$ nitrogen (molar or volume proportions), the molar fractions of hydrogen, oxygen and nitrogen in the considered stoichiometric mixture are thus equal to $2/7$, $1/7$ and $4/7$ respectively.
The mass fractions of these constituents are thus easily deduced from these values:
\[
y_{\mathrm{H}_2} = \frac{2 \,W_{\mathrm{H}_2}}{W_t},\quad
y_{\mathrm{O}_2} = \frac{W_{\mathrm{O}_2}}{W_t},\quad
y_{\mathrm{N}_2} = \frac{4 \,W_{\mathrm{N}_2}}{W_t},\quad
W_t=2 W_{\mathrm{H}_2}+W_{\mathrm{O}_2}+4\,W_{\mathrm{N}_2}
\]
where $W_{\mathrm{H}_2}=0.002$\,Kg, $W_{\mathrm{O}_2}=0.032$\,Kg and $W_{\mathrm{N}_2}=0.028$\,Kg stand for the molar mass of the hydrogen, oxygen and nitrogen molecules respectively.
Since $\mathrm{H}_2$ and $\mathrm{O}_2$ are pure substances (and thus their formation enthalpy is equal to zero), the chemical reaction heat reads:
\[
Q = y_{\mathrm{H}_2\mathrm{O}}\ \Delta H_0^f = (y_{\mathrm{H}_2}+y_{\mathrm{O}_2})\ \Delta H_0^f,
\]
where $\Delta H_0^f=1.3255\,10^7$ J/Kg stands for the formation enthalpy of steam.
The initial pressure is $p=10^5$\,Pa, the initial temperature is $T=283\,^\circ$K, the initial density is given by the Boyle-Mariotte law and the heat capacity ratio is $\gamma_u = \gamma_b = 1.4$.

\medskip
We plot on Figures \ref{fig:uf=32} and \ref{fig:uf=4} the density, velocity and pressure profiles obtained for $u_f=32$\,m/s and $u_f=4$\,m/s respectively.
The solution in the intermediate zone is obtained with a regular mesh, splitting the interval between $5$\,m and the position of the precursor shock (which is unknown up to the last solution step of the algorithm) in $5000$ equal subintervals.
Then we show on Figures \ref{fig:rho_1}-\ref{fig:pressures} the evolution of the states $W_1$, $W_2$ and $W_b$ as a function of the flame velocity.
The temperature in the burnt state, not shown here, is close to $T=3050\,^\circ$K for all the values of the flame velocity $u_f$.
We observe that the precursor shock is of very weak amplitude for low values of the flame velocity; in fact, it becomes visible only when $u_f$ reaches $20$\,m/s.
For $u_f=4$, the computed velocity at state $W_1$ is lower than $10^{-6}$\,m/s, while it reaches values greater than $30$\,m/s at State $W_2$.
Since the ordinary differential equation governing the velocity in the intermediate zone \eqref{eq:ode_u} is of the form
\[
u'= f(\rho,u)\ u,
\]
one may anticipate such a low value as initial (right) condition to lead to severe accuracy problems.
In this respect, the computed value which seems to be the most affected is the velocity at state $W_2$: the convergence value seems to be close to $33.00$\,m/s, we obtain $u_2\simeq 34.5$\,m/s with $n=5\, 10^3$ cells and $u_2 \in (32.95\,\mathrm{m/s},\ 33\,\mathrm{m/s})$ for $n=8\, 10^4$, $n=16\, 10^4$, $n=32\, 10^4$ and $n=64\, 10^4$.
As expected, convergence is easier when the precursor shock has a significant amplitude: $u_2\simeq 243.0$\,m/s for $n=5000$, for a convergence value in the range of $243.8$\,m/s.

\begin{figure}[tb]
\scalebox{0.9}{
\begin{tikzpicture} 
\node[anchor=east] at (10., 0){\includegraphics[scale=0.6]{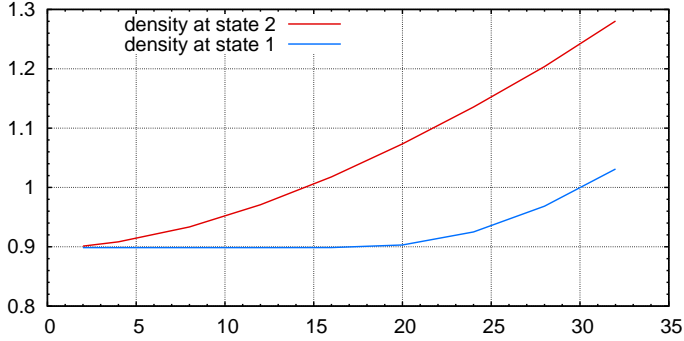}};
\end{tikzpicture}
}
\caption{Density at state 1 and state 2 as a function of the flame velocity.}
\label{fig:rho_1} 
\end{figure}

\begin{figure}[tb]
\scalebox{0.9}{
\begin{tikzpicture} 
\node[anchor=east] at (10., 0){\includegraphics[scale=0.6]{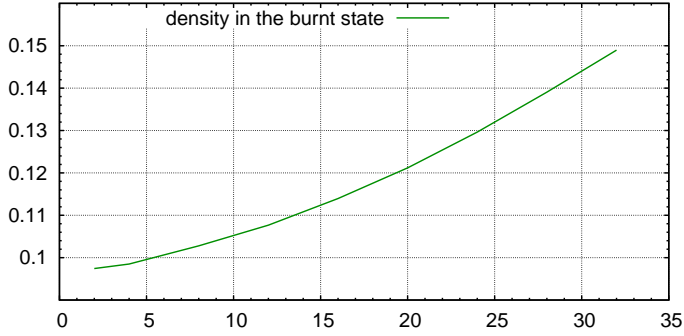}};
\end{tikzpicture}
}
\caption{Density in the burnt zone as a function of the flame velocity.}
\label{fig:rho_2} 
\end{figure}

\begin{figure}[tb]
\scalebox{0.9}{
\begin{tikzpicture} 
\node[anchor=east] at (10., 0){\includegraphics[scale=0.6]{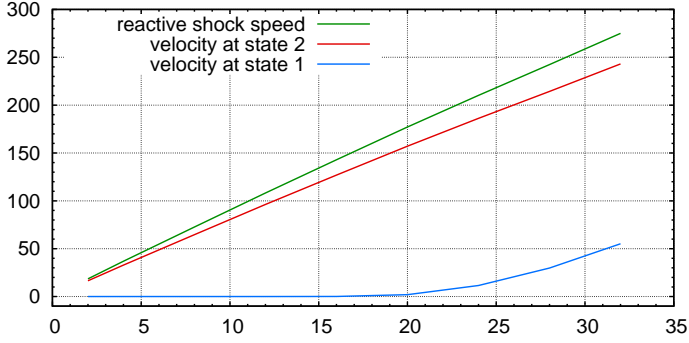}};
\end{tikzpicture}
}
\caption{Velocity at state 1 and state 2 and speed of the precursor shock as a function of the flame velocity.}
\label{fig:velocities} 
\end{figure}

\begin{figure}[tb]
\scalebox{0.9}{
\begin{tikzpicture} 
\node[anchor=east] at (10., 0){\includegraphics[scale=0.6]{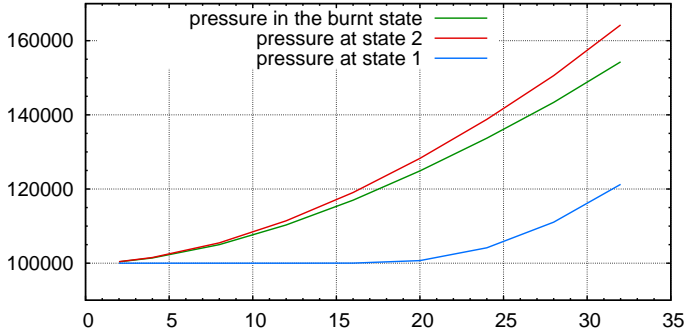}};
\end{tikzpicture}
}
\caption{Pressure at state 1, at state 2 and in the burnt zone as a function of the flame velocity.}
\label{fig:pressures} 
\end{figure}
%
% -----------------------------------------------------------------------------------------------------------------------------------------------------------
%
\appendix
\ifthenelse{\boolean{@long}}{
\section{Euler equations in spherical coordinates}
The Euler equations read in cartesian coordinates:
\begin{subequations}\label{eq:euler_cart}
\begin{align} &
\partial_t \bar{\rho} + \dive (\bar{\rho} \bar{\bfu}) = 0,
\\ &
\partial_t (\bar{\rho}\bar{\bfu}) + \dive (\bar{\rho} \bar{\bfu} \otimes \bar{\bfu}) + \nabla \bar{p} = 0,
\\ &
\partial_t (\bar{\rho} \bar{E}) + \dive (\bar{\rho} \bar{E} \bar{\bfu} + \bar{p} \bar{\bfu})= 0,
\end{align} \end{subequations}
where $\bar{\rho} = \bar{\rho} (t, \bfx) \in \xR$ the density, $\bar{\bfu}= \bar{\bfu} (t, \bfx) \in \xR^3$ the velocity, $\bar{p} = \bar{p} (t,\bfx) \in \xR$ the pressure and $\bar{E} = \bar{E} (t, \bfx) \in \xR$ the total energy for all $ t \in \xR$ and $\bfx = (x_1, x_2, x_3) \in \xR^3$.
This system is closed by the equation of state, which for a perfect gas, is given by
\begin{equation}\label{eq:eos}
\bar{E}= \frac{1}{2}|\bar{\bfu}|^2 + \bar{e}, \text{ with } \bar{p}  = (\gamma -1) \bar{\rho} \bar{e},
\end{equation}
where $\bar e = \bar e(t, \bfx)$ the internal energy and $\gamma > 1$ the heat capacity ratio.
We suppose that the flow satisfies a spherical symmetry assumption, so the solution of equations \eqref{eq:euler_cart}-\eqref{eq:eos} may be recast as:
\begin{equation}\label{sp}
\bar{\rho}(t,\bfx)=\rho (t,r), \, \bar{p}(t, \bfx)=p(t,r),\, \bar{E}(t,\bfx)=E(t,r) \text{ and } \bar{\bfu}(t,\bfx)=u(t, r) \frac{\bfx}{r},
\end{equation}
with $r = |\bfx|$ and where $(\rho, u, p, E)(t,r) \in \xR^4$ are scalar functions, \ie\ $(\rho, u, p, E) \in \xR^4$.
The aim of this section is to derive the system of equations satisfied by $(\rho, u, p, E)$.
We suppose first that these functions are regular, so we obtain the strong form of the so-called Euler equations in spherical coordinates; then we turn to the weak form, valid for shock solutions.
%
% -----------------------------------
%
\subsection{Regular solutions}
Let us use the notation $\partial _i = \dfrac{\partial}{\partial x_i}$.
We begin by deriving the following three identities:
\begin{itemize}
\item[$(a)$] $\displaystyle \partial _i r = \frac{x_i}{r}$, for $i=1,\ 2,\ 3$.
\item[$(b)$] If $f = f(r)$,  $\displaystyle \dive (f \bar{\bfu}) = \frac{1}{r^2} \partial _r (r^2 f u )$.

\medskip
\item[$(c)$] If $f = f(r)$, $\displaystyle \gradi f = \partial _r f\ \frac{\bfx}{r}$.
\end{itemize}
The first item is a straightforward consequence of the definition $r=|\bfx|$.
For Item $(b)$, we have, thanks to $(a)$,
\begin{eqnarray*}
\dive(f \bar{\bfu}) & = &
\sum_{i=1}^3 \partial_i (f u \dfrac{x_i}{r})
\\ & = &
f u \sum_{i=1}^3 \partial_i  (\dfrac{x_i}{r}) + \sum_{i=1}^3 \dfrac{x_i}{r}\, \partial_i (f u)
\\ & = &
f u \sum_{i=1}^3 (\dfrac{1}{r} - \dfrac{x_i^2}{r^3}) + \sum_{i=1}^3 \dfrac{x_i}{r}\, \partial_ir\ \partial_r(fu)
\\ & = &
\dfrac{1}{r} f u \sum_{i=1}^3 (1 - \dfrac{x_i^2}{r^2}) + \partial_r(fu)\ \sum_{i=1}^3 \dfrac{x_i^2}{r^2}
\\ & = &  
\dfrac{2}{r} fu + \partial_r(f u)
= \dfrac{1}{r^2}  \partial_r (r^2 fu).
\end{eqnarray*}
Item $(c)$ is an immediate consequence of $(a)$.

\medskip
We are now in position to state the following lemma.

\begin{lemma}
Suppose that $(\bar{\rho}, \bar{\bfu}, \bar{p}, \bar{E})$ is solution of \eqref{eq:euler_cart}; then $(\rho, u, p, E)$ satisfies:
\begin{subequations}\label{eq:euler_sph}
\begin{align} \label{eq:euler_sph_mass} &
\partial_t (r^2 \rho ) + \partial_r (r^2 \rho u ) =0,
\\ \label{eq:euler_sph_mom} &
\partial_t (r^2\rho u) +  \partial_r (r^2(\rho u^2+p)) = 2r p,
\\ \label{eq:euler_sph_E} &
\partial _t (r^2\rho E ) + \partial _r (r^2(\rho uE + p u )) = 0,
\end{align}\end{subequations}
with $E= \frac 1 2\, u^2 + e$ and $p=(\gamma -1) \rho e$.
\end{lemma}

\begin{proof}
The mass balance equation \eqref{eq:euler_sph_mass} is a straightforward consequence of Item $(b)$.
For the momentum balance equation, we first remark that, for any function $f(r)$, we have:
\[
\dive (f x_i \bar{\bfu}) = \frac{x_i}{r^2} \partial _r (r^2 f u) + \frac{x_i }{r} f u.
\]
Indeed, this relation follows from the development $\dive (f x_i \bar{\bfu}) = x_i \dive (f \bar{\bfu}) + f \bar{\bfu} \cdot \nabla x_i$ thanks to Item $(b)$.
Applying this identity with $f= \displaystyle \frac{\rho u}{r}$, we thus obtain 
\begin{multline*}
\dive(\rho u_i \bfu) =
\dive( \frac{\rho u }{r} x_i \bfu ) = \frac{x_i}{r^2}\, \bigl( \partial _r (\rho u^2 r) + \rho u^2 \bigr)
\\
= \frac{x_i}{r^2}\, \bigl( r\partial _r (\rho u^2) + 2 \rho u^2 \bigr) =  \frac{x_i}{r^3} \partial _r ( r^2 \rho u^2).
\end{multline*}
Thanks to this relation and using $(c)$ for the pressure gradient, for $i=1,\ 2,\ 3$, the $i-$th component of the momentum equation reads:
\[
\dfrac{x_i}{r} \big[ \partial_t (\rho u ) +  \frac{1}{r^2}\partial_r (r^2 \rho u^2)+ \partial_r p \big]= 0.
\]
The three components of this vector balance equation thus boil down to the single relation:
\[
\partial_t (\rho u ) + \frac{1}{r^2}\partial_r (r^2 \rho u^2)+ \partial_r p = 0.
\]
Multiplying by $r^2$, we obtain the conservative form \eqref{eq:euler_sph_mom}.
For the energy balance equation, we first note that $\bar{E} = \frac 1 2\, |\bar{\bfu}|^2 + \bar{e} = \frac 1 2\, u^2 + e = E$.
Then, using once again $(b)$, we get
\[
\dive(\bar{\rho} \bar{E} \bar{\bfu}) = \frac{1}{r^2} \partial_r(\rho E u r^2)
\text{ and } \dive(\bar{p} \bar{\bfu}) = \frac{1}{r^2} \partial_r(p u r^2),
\]
and \eqref{eq:euler_sph_E} follows by adding the time derivative of $(\bar{\rho} \bar{E})$.
\end{proof}

\medskip
We can apply the same process for the entropy balance equation.
In the Cartesian system of coordinates, this relation reads, for regular solutions:
\begin{equation}\label{eq:euler_s}
\partial_t (\bar{\rho} \bar{s}) + \dive (\bar{\rho} \bar{\bfu} \bar{s}) = 0,
\quad \text{with } \bar{s} = \dfrac{\bar{p}}{\bar{\rho}^{\gamma}},
\end{equation}
with $\bar s = \bar s (t,\bfx)$.
By this latter expression, under spherical symmetry assumption, there exists a function $s(t,r)$ such that $\bar s (t,\bfx)= s(t,r)$.
This latter function satisfies
\[
s=\dfrac{p}{\rho^\gamma},
\]
and a straightforward application of Identity $(b)$ with $f=\rho s$ yields, from \eqref{eq:euler_s}:
\begin{equation}\label{eq:euler_sph_s}
\partial_t ( r^2 \rho s) + \partial_r (r^2 \rho s u) =  0.
\end{equation}
%
% -----------------------------------
%
\subsection{Weak solutions}
Let us now treat the case of no classical solution of (\ref{Er}) and for the sake of simplicity we only consider the mass balance equation for a velocity $\bar{\bfu}$ given by (\ref{sp}). So $\bar{\rho}$ is weak solution of $\partial_t \bar{\rho} + \dive (\bar{\rho} \bar{\bfu}) = 0$, if $\forall \bar \varphi \in C_c^{\infty}(\xR_+ \times \xR^d, \xR)$, we have
\begin{equation}\label{weak-sol}
\int_{\xR_+} \int_{\xR} \bar{\rho} \partial_t \bar \varphi + \bar{\rho} \bar{\bfu} \cdot\nabla \bar \varphi \, d\bfx \,dt + \int_{\xR^d} \bar{\rho}_0(\bfx) \bar \varphi(0, \bfx) \, d\bfx = 0,
\end{equation}
with $ \bar{\rho}_0(\bfx)  = \bar{\rho}(0, \bfx).$ 
\begin{lemma}
Let $\Omega$ a open set of $\xR_+^*$ et let's consider $\bar{\rho}$, given by (\ref{sp}), a weak solution of (\ref{weak-sol}) in the form
\begin{equation}\label{choc-sol}
\rho(t, r) = 
\begin{cases}
\rho_L \qquad \text{ if } r \in \Omega^{-} \\
\rho_R \qquad \text{ if } r \in \Omega^{+}
\end{cases}
\text{ with }
u(t, r) = 
\begin{cases}
u_L \qquad \text{ if } r \in \Omega^{-} \\
u_R \qquad \text{ if } r \in \Omega^{+}
\end{cases}
\end{equation}
where $\Omega^{-} = \lbrace r \in \Omega, r \leq t \sigma\rbrace $ and $\Omega^{+} = \lbrace r \in \Omega, r > t \sigma\rbrace$ with $(\rho_R, u_R)$ and  $(\rho_L, u_L)$ are constant in the domains $\xR_+ \times \Omega^{+}$ and $\xR_+ \times \Omega^{-}$ respectively. \\
Let $ \bar \varphi \in C_c^{\infty}(\xR_+ \times \xR^d, \xR)$ and $\varphi \in C_c^{\infty}(\xR_+ \times \xR^*_+, \xR)$ such that $\bar \varphi(t, \bfx) = \varphi(t, r)$, then $\rho$ satisfies the following weak formulation 
\begin{equation}\label{weak-sol-rad}
\int_{\xR_+} \int_{\xR_+^*} (r^2 \rho \partial_t \varphi + r^2 \rho u \partial_r \varphi )\,dr \, dt + \int_{\xR_+^*} r^2\rho_0(r)  \varphi(0, r) \, dr  = 0.
\end{equation}
Furthermore $\sigma$ checks the relationship: 
\begin{equation}\label{cond-R.H}
\sigma (\rho_R - \rho_L) = (\rho_R u_R - \rho_Lu_L).
\end{equation}
\end{lemma}
\begin{proof}
Let $ \bar \varphi \in C_c^{\infty}(\xR_+ \times \xR^d, \xR)$ and let introduce the function $\varphi$ such that $\bar \varphi(t, \bfx) = \varphi(t, r)$, then $\varphi \in C_c^{\infty}(\xR_+ \times \xR_+^*, \xR)$. 
Therefore
\begin{align*}
\int_{\xR_+} \int_{\xR^d} \bar{\rho} (\partial_t  \bar \varphi &+ \bar{\rho} \bar{\bfu} \cdot \nabla \bar \varphi) \, d\bfx \,dt 
\\
& = \int_{\xR_+} \int_{\xR_+} ( \rho(t, r) \partial_t \varphi(t,r) + \rho(t,r) u(t, r) \dfrac{\bfx}{r}  \cdot \nabla \varphi(t,r) ) r^2 \,dr \,dt \\   &= \int_{\xR_+} \int_{\xR_+} (r^2 \rho(t,r) \partial_t \varphi(t, r) + r^2 \rho(t,r) u(t, r) \dfrac{\bfx}{r} \cdot \dfrac{\bfx}{r} \partial_r \varphi(t, r))  \,dr \,dt \\
  &=  \int_{\xR_+} \int_{\xR_+} (r^2 \rho(t,r) \partial_t \varphi(t, r) + r^2 \rho(r,t) u(t, r)  \partial_r \varphi(t, r))  \,dr \,dt.
\end{align*}
We have also that
\[
\int_{\xR^d} \bar{\rho}_0(\bfx) \bar \varphi(0, \bfx) \, d\bfx = \int_{\xR_+} r^2 \rho_0(r)  \varphi(0, r) \, dr
\]
Thus (\ref{weak-sol}) reads: 
\begin{multline}
\int_{\xR_+} \!\int_{\xR_+} \!(r^2 \rho(t,r) \partial_t \varphi(t, r) + r^2 \rho(r,t) u(t, r)  \partial_r \varphi(t, r))  \ dr \ dt + \int_{\xR_+}\! r^2 \rho_0(r)  \varphi(0, r) \ dr = 0, \\ \,\forall \varphi \in C_c^{\infty}(\xR_+ \times \xR_+, \xR).
\end{multline}
According to the definition (\ref{choc-sol}) and decomposing the integral into space on $\Omega^{+}$ and $\Omega^{-}$, we obtain that
\begin{equation*}
\int_{\xR_+} \int_{\Omega^{+}} r^2\rho_R \partial_t \varphi +r^2 \rho_R u_R \partial_r \varphi \,dr\,dt + \int_{\xR_+} \int_{\Omega^{-}} r^2\rho_L \partial_t \varphi + r^2\rho_L u_L \partial_r \varphi \,dr\,dt = 0.
\end{equation*}
Thus as in the scalar case, we can show that, $\sigma$ satisfies the Rankine Hugoniot relationship
\[
\sigma (r^2 (\rho_R - \rho_L)) = (r^2(\rho_R u_R - \rho_Lu_L)).
\]
The lemma is thus proved.
\end{proof}
%
% -----------------------------------------------------------------------------------------------------------------------------------------------------------
%
\section{Left state of a shock as a function of the right state and the shock velocity}
In this section, we recall a classical computation which consists in determining the left state of a shock as a function of the right state and the shock velocity.

\medskip
\begin{lemma} \label{lmm:usual_shock}
Let $W=(\rho,u,p)$ be the left state of a shock travelling at the given speed $\sigma$.
Let $W_R=(\rho_R,u_R,p_R)$ be the given right state, which is supposed to satisfy $u_R=0$.
Let $c_R$ be the speed of sound in the right state,\ie\ $c_R^2=\gamma p_R/c_R$, and let $M$ be the Mach number associated to the incident shock, defined by $M=\sigma/c_R$.
Then $W$ is given by:
\begin{subequations}\label{eq:shock}
\begin{align} \label{eq:sh_rho} &
\rho= \frac{\gamma +1}{\gamma -1+\dfrac 2 {M^2}}\ \rho_R,
\\ \label{eq:sh_u} &
u=(1-\frac{\rho_R}\rho)\, \sigma,
\\ \label{eq:sh_p} &
p=p_R+(1-\frac{\rho_R}\rho)\rho_R\,\sigma^2.
\end{align} \end{subequations}
\end{lemma}

\begin{proof}
We first change the coordinate system, in such a way that the shock is steady in the new coordinate system.
The density and the pressure are left unchanged, while the velocity is now $u-\sigma$ and $-\sigma$ in the left and right state respectively.
In this coordinate system, the Rankine-Hugoniot conditions imply that the jump of the fluxes vanishes, which reads for the Euler equations:
\begin{subequations}\label{eq:RHs}
\begin{align} \label{eq:RHs_mass} &
\rho\,(u-\sigma)=\rho_R\,(-\sigma),
\\[1ex] \label{eq:RHs_mom} &
\rho\,(u-\sigma)^2 + p = \rho_R\,(-\sigma)^2 + p_R,
\\[1ex] \label{eq:RHs_e} &
\frac 1 2 \rho\, (u-\sigma)^3 + \rho e\,(u-\sigma) + p \,(u-\sigma) = \frac 1 2 \rho_R\, (-\sigma)^3 + \rho_R e_R\,(-\sigma) + p_R \,(-\sigma).
\end{align}\end{subequations}
This system must be complemented by the equation of state $p=(\gamma -1)\rho e$ and $p_R=(\gamma -1)\rho_R e_R$.
Thanks to this relation, we may recast \eqref{eq:RHs_e} as:
\begin{equation} \label{eq:RHs_ed}
\rho\,(\sigma-u)^3+\xi\,p\,(\sigma-u)=\rho_R \sigma^3+\xi\,\sigma\, p_R,\quad \xi=\frac{2\gamma}{\gamma -1}.
\end{equation}
Relation \eqref{eq:RHs_mom} reads
\[
\frac 1 \rho \bigl[\rho\,(u-\sigma)\bigr]^2 + p = \rho_R\,(-\sigma)^2 + p_R,
\]
so, using \eqref{eq:RHs_mass}:
\[
\frac 1 \rho (\rho_R\, \sigma)^2 + p = \rho_R\,\sigma^2 + p_R.
\]
We thus obtain $p$ as a function of known quantities (\ie\ $\sigma$ and the right state) and $\rho$ only:
\begin{equation}\label{eq:sh_pd}
p = p_R + (1 - \frac{\rho_R} \rho)\,\sigma^2.
\end{equation}
We now notice that substituting, in the jump condition associated to the energy balance \eqref{eq:RHs_ed}, this expression for $p$ and $(\rho_R/\rho)\,\sigma$ for $(u-\sigma)$, thanks once again to Equation \eqref{eq:RHs_mass}, we get an equation for $\rho$ only:
\[
\rho\,\bigl(\frac{\rho_R} \rho\,\sigma\bigr)^3+\xi\, \Bigl( p_R + (1 - \frac{\rho_R} \rho)\,\sigma^2 \Bigr)\, \bigl(\frac{\rho_R} \rho\,\sigma\bigr)
=\rho_R \sigma^3+\xi\,\sigma.
\]
If $\sigma=0$, the first jump condition implies $u=0$ (excluding $\rho=0$), then the second one yields $p=P_R$ and the third one is automatically satisfied: the considered discontinuity is a (stationary) contact.
In such a case, the right state remains partially undermined by the jump conditions: $\rho$ and $e$ may take any value satisfying $(\gamma -1)\rho e =p_R$.
If we only consider a shock, $\sigma \neq 0$ and the last relation may be simplified by $\sigma$.
Reordering, we get:
\[
\rho_R\,\sigma^2 \Bigl((\frac{\rho_R}\rho)^2 -1 \Bigr)+\xi\,p_R\,(\frac{\rho_R}\rho -1) - \xi \rho_R\,\sigma^2 \frac{\rho_R}\rho (\frac{\rho_R}\rho -1)=0.
\]
The case $\rho=\rho_R$ has no interest: it yields, by the first jump condition, $u=0$, and the second one implies that $p=p_R$, which means {\em in fine} that $W=W_R$, \ie\ that there is no discontinuity at all.
We may thus simplify by $1-\rho_R/\rho$, to obtain a linear equation for the ratio $\rho_R/\rho$.
Solving this latter equation, we obtain:
\[
\rho=\frac{\gamma +1}{\gamma -1+\dfrac 2 {M^2}}\ \rho_R,\quad \mbox{with } M=\frac \sigma {c_R},\ c_R^2=\gamma \frac{p_R}{\rho_R}.
\]
We thus obtain \eqref{eq:sh_rho}.
Relation \eqref{eq:sh_u} is a straightforward consequence of \eqref{eq:RHs_mass} and \eqref{eq:sh_p} was already proven (Relation \eqref{eq:sh_pd} below).
The proof is thus complete.\\
\end{proof}

\medskip
For a 3-shock, entropy conditions requires that $\sigma > c_R$, which is equivalent to $M>1$.
We thus have $\rho > p_R$, $p > p_R$ and $0 < u < \sigma$.

\medskip
It is worth noting that it is now easy to relax the assumption $u_R=0$ of Lemma \ref{lmm:usual_shock}.
Indeed, for the general case, we may work in the system of coordinates in translation at the velocity $u_R$ with respect to the initial one: in the new coordinate system, the right state is now at rest and Lemma \ref{lmm:usual_shock} applies, replacing $\sigma$ by $\sigma -u_R$ and $u$ by $u-u_R$ in the definition of the Mach number $M$ and in the system of equations \eqref{eq:shock}.
}
{}
%
% -----------------------------------------------------------------------------------------------------------------------------------------------------------
%
\section{A relation satisfied by the right state of a shock when the left state is at rest}\label{sec:RS}
In this section, we perform a technical computation motivated by the following arguments.
Let us suppose that a shock travelling at the speed $\sigma_r$ separates a left state denoted by $W_b=(\rho_b,u_b,p_b)$ and a right state denoted by $W_2=(\rho_2,u_2,p_2)$, and that $u_b=0$.
The Rankine-Hugoniot conditions yield 3 independent equations, and thus constitute a system in which $\rho_b$ and $p_b$ may be eliminated, to obtain a relation linking $W_2$ and $\sigma_r$ only.
It is this relation that we now derive, supposing that the following specific constitutive relations hold for both states:
\begin{equation} \label{app:eos}
E_b=\frac 1 2 u_b^2 + e_b +Q,\ p_b=(\gamma_b-1)\,\rho_b\, e_b,\quad
E_2=\frac 1 2 u_2^2 + e_2,\ p_2=(\gamma_u-1)\,\rho_2\, e_2.
\end{equation}
Note that eliminating $\rho_b$ and $p_b$ consists in establishing an expression of these quantities (and thus of $W_b$, since $u_b=0$) as a function of $W_2$ and $\sigma_r$.
All of these relations, \ie\ the equation linking $W_2$ and $\sigma_r$ and the expression of $W_b$ as a function of these variables, are gathered in the following lemma.

\medskip
\begin{lemma}[Some conditions at the reactive shock]
The state $W_2$ and the shock speed $\sigma_r$ satisfy the following relation
\begin{equation}\label{eq:app_etat-2}
\frac12 u_2^2 + \frac 1 {\gamma_b -1} u_2 \sigma_r + \bigl( \frac{\gamma_u}{\gamma_u -1} - \frac{\gamma_b}{\gamma_b -1} \dfrac{\sigma_r}{\sigma_r - u_2} \bigr) \dfrac{p_2}{\rho_2} + Q =0.
\end{equation}
In addition, $W_b$ is given as a function of $W_2$ and $\sigma_r$ by:
\begin{equation} \label{eq:app_etatb}
\rho_b = \rho_2 (\dfrac{\sigma_r - u_2}{\sigma_r}),\quad
p_b = p_2 - \rho_2 u_2 (\sigma_r - u_2).
\end{equation}
\end{lemma}

\begin{proof}
Using the standard change of coordinates to work in the coordinate system in which the shock is at rest (see \eg\ \cite{god-96-num}), the Rankine-Hugoniot relationships (which boil down to the fact that the jump of the fluxes vanishes) through the shock for the mass and momentum balance equations read respectively:
\[
\rho_b\, \sigma_r = \rho_2\, (\sigma_r - u_2),\quad \rho_b\, \sigma_r^2 + p_b = \rho_2 (\sigma_r - u_2)^2 + p_2.
\]
These two relations readily yields the expression \eqref{eq:app_etatb} of $W_b$ as a function of $W_2$ and $\sigma_r$ that we are looking for:
\[
\rho_b = \rho_2\, (\dfrac{\sigma_r - u_2}{\sigma_r}) \text{ and }  p_b = p_2 - \rho_2\, u_2\, (\sigma_r - u_2).
\]
Let us now write the Rankine-Hugoniot condition for the conservation equation of the total energy:
\begin{equation}\label{re}
\rho_b\, \sigma_r\, E_b + \sigma_r\, p_b = \rho_2\, (\sigma_r -u_2)\, E_2 + (\sigma_r - u_2)\, p_2.
\end{equation}
We may divide the left-hand side of his relation by $\rho_b \sigma_r$ and the right-hand side by $\rho_2 (\sigma_r - u_2)$ (since these two expressions are equal by the jump condition associated to the mass balance equation), to obtain:
\[
E_b +  \dfrac{p_b}{\rho_b} = E_2 + \dfrac{p_2}{\rho_2}.
\]
Using the constitutive relations \eqref{app:eos} and the expression \eqref{eq:app_etatb} of $\rho_b$ and $p_b$ in this equation yields \eqref{eq:app_etat-2} and thus concludes the proof.\\
\end{proof}
%
% -----------------------------------------------------------------------------------------------------------------------------------------------------------
%
\ifthenelse{\boolean{@long}}{
\section{A technical lemma}\label{sec:tl}
We prove in this section the technical lemma \ref{lmm:tl}, which we first recall.

\smallskip
\begin{lemma}
Let $h$ be a continuously differentiable real function, let us suppose that there exists $a > 0$ such that $h(a) > 0$, and that $h$ satisfies the property $h'(x) \leq 0$ if $h(x) > 0$.
Then $h(x) \geq h(a)$, for all $x < a$.
\end{lemma}
\begin{proof}
The result is proven by contradiction: assume that there exists $b < a$ such that $h(b) < h(a)$.
Then by the mean value theorem, there exists $c \in (b,a)$ such that 
\[
h'(c) = \frac{h(a)-h(b)}{a-b}>0.
\]
Hence, thanks to the assumption of the lemma, $h(c) \le 0$. 
By continuity, there exists $d \in [c,a]$ such that $h(d) =0$.
Let $\underline d = \max \{d  \in \xR \ : \ d < a  \mbox{ and } h(d) = 0\}$; then
\[
\int_{\underline d}^a h'(t) dt = h(a) - h(\underline d) \le 0
\]
which leads to a contradiction, since $h(a) >0$ and $h(\underline d) = 0$.
\end{proof}
}
{}
%
% -----------------------------------------------------------------------------------------------------------------------------------------------------------
%
\section{Variations of $\mathcal F_r$ when $\gamma_b=\gamma_u$}\label{sec:fr}
In this appendix, we prove that the function $\mathcal F_r$ defined by \eqref{eq:fr} is increasing over $(\sigma_\ell,\sigma_p)$ if the heat capacity ratios $\gamma_b$ and $\gamma_u$ are equal.
Let us denote by $\gamma$ the common heat capacity ratio, and recall the expression of $\mathcal F_r$:
\begin{equation}\label{eq:fr}
\mathcal F_r(x)=
\frac12 u(x)^2 + \frac 1 {\gamma -1} x\,u(x) - \frac \gamma {\gamma -1}\, \dfrac{u(x)}{x - u(x)}\, \dfrac{p(x)}{\rho(x)} + Q.
\end{equation}
Using $\gamma p/\rho=c^2$, we have
\begin{multline*} 
\mathcal F_r'(x) = u(x) u'(x) + \frac x {\gamma \!-\!1} u'(x) +  \frac 1 {\gamma \!-\!1} u(x)
\\
-\frac 1 {\gamma \!-\!1}\, \frac{x u'(x) - u(x)}{(x - u(x))^2}  c^2(x)
-\frac 1 {\gamma \!-\!1} \, \frac {u(x)} {x - u(x)}  (c^2)'(x),
\end{multline*}
with 
\[
(c^2)' = s_1 \gamma (\rho^{\gamma -1})' = s_1 \gamma (\gamma-1)\, \rho^{\gamma -2} \rho'= (\gamma -1)\, c^2\, \dfrac{\rho'} \rho.
%= - 2 \dfrac{(\gamma_u -1) u (u-x)}{x \big((u - x)^2 - c^2 \big)} c^2.
\]
So we get that $\mathcal F_r'(x)=T_1(x)+T_2(x)+T_3(x)$ with
\[\begin{array}{l}\displaystyle
T_1(x)= \frac 1 {\gamma -1} \bigl(1 + \frac{c^2(x)}{(x - u(x))^2} \bigr) u(x),
\\[2ex] \displaystyle
T_2(x)= \Bigl(u(x) +\frac x {\gamma -1} - \dfrac 1 {\gamma-1} \frac{ x c^2(x)}{(x - u(x))^2} \Bigr) u'(x),
\\[2ex] \displaystyle
T_3(x)=- \frac {u(x) \, c^2(x)} {(x - u(x))\,\rho}\,\rho'.
\end{array}\]
Since $u>0$ by Lemma \ref{lmm:var}, $T_1 >0$.
In addition, in the proof of the same lemma \ref{lmm:var}, we showed that $u(x)+c(x)-x >0$, so that $c(x) > x-u(x)$ and
\[
u(x) +\frac x{\gamma -1}(1 - \dfrac{  c^2(x)}{(x - u(x))^2})  \leq  u(x).
\]
Since $u'\leq 0$, this implies that $T_2 \geq u u'$.
Replacing $\rho'$ and $u'$ by their expressions given in \eqref{eq:ode}, we obtain that $u u' + T_3=0$, which concludes the proof.

\bibliographystyle{plain}
\bibliography{sedov}
\end{document}